\documentclass[12pt]{amsart}
\usepackage{color}
\usepackage{amssymb}
\makeatletter
\@namedef{subjclassname@2020}{\textup{2020} Mathematics Subject Classification} 
\makeatother
\newcommand{\Ric}{\textit{Ric}}

\newcommand{\tr}{\operatorname{tr}}
\newcommand{\RPn}{\mathbb R \mathbb P^n}

\newtheorem{theorem}{Theorem}[section]
\newtheorem{lemma}[theorem]{Lemma}
\newtheorem{proposition}[theorem]{Proposition}
\newtheorem{corollary}[theorem]{Corollary}

\theoremstyle{definition}
\newtheorem{definition}[theorem]{Definition}
\newtheorem{example}[theorem]{Example}

\theoremstyle{remark}
\newtheorem{remark}[theorem]{Remark}

\numberwithin{equation}{section}

\begin{document}
\title{On a constant curvature statistical manifold}
\author{Shimpei Kobayashi$^*$}
\address{Department of Mathematics, Hokkaido University, 
Sapporo, 060-0810, Japan}
\email{shimpei@math.sci.hokudai.ac.jp}
\thanks{The first named author is partially supported by Kakenhi 
18K03265. }
\keywords{Statistical manifolds; constant curvatures; conjugate symmetries; projective 
 flatness; properly convex structures}
\author{Yu Ohno}
\email{floodcaptain@gmail.com}
\thanks{}
\subjclass[2020]{Primary:~53B12,~53C15}

\dedicatory{}
\begin{abstract}
 We will show that a statistical manifold $(M, g, \nabla)$ has a constant curvature 
 if and only if it is a projectively flat conjugate symmetric manifold, 
 that is, the affine connection $\nabla$ is projectively flat 
 and the curvatures satisfies $R=R^*$, where $R^*$ is the curvature 
 of the dual connection $\nabla^*$.
 Moreover, we will show that properly convex structures 
 on a projectively flat compact manifold
 induces constant curvature $-1$ statistical structures and vice versa.
\end{abstract}

\maketitle
\section*{Introduction}
 On a pseudo-Riemannian manifold $(M, g)$, consider a torsion-free 
 affine connection $\nabla$
 on $(M, g)$ and a $(0, 3)$-tensor field $C$ defined by 
 \begin{equation}\label{eq:C}
 C (X, Y, Z) = (\nabla_X g) (Y, Z).
 \end{equation}
 This pair $(g, \nabla)$ is called a \textit{statistical structure}
 if $C$ is totally symmetric, and the tensor $C$ will be called  
 the \textit{cubic form} or the \textit{cubic tensor}. 
  A pseudo-Riemannian manifold $(M, g)$ with a statistical structure $(g, \nabla)$ will be called 
 the \textit{statistical manifold} and it will be denoted by a triad $(M, g, \nabla)$.
 It is important to consider the \textit{dual} torsion-free affine connection $\nabla^*$
 for a statistical manifold 
 $(M, g, \nabla)$ defined by 
\begin{equation}\label{eq:dualnablas}
  X g (Y, Z ) = g (\nabla_X Y, Z) + g (Y, \nabla^*_X Z). 
\end{equation}
 Note that $\nabla = \nabla^*$ if and only if $\nabla$ is the Levi-Civita connection of the metric $g$. Thus the statistical manifold is a natural generalization of a 
 pseudo-Riemannian manifold with the Levi-Civita connection, which 
 gives a trivial statistical structure. For the theory of statistical manifolds, 
 we refer the readers to \cite{AN}.

 For a statistical manifold $(M, g, \nabla)$ let $R$ denote the curvature tensor field of  $\nabla$, and it is said to be of \textit{constant curvature} $k$ if 
 \begin{equation}\label{eq:constantcurv}
 R(X, Y) Z = k \left\{ g(Y, Z) X - g(X, Z)Y \right\}
 \end{equation}
 holds for any vector fields $X, Y$ and $Z$ and some real constant $k$, 
  see \cite{Kur}. 
 The statistical manifold $(M, g, \nabla)$ is 
 said to be \textit{conjugate symmetric} if
\begin{equation}\label{eq:RR*}
 R = R^* 
\end{equation}
 holds,  where $R^*$ is the curvature of $\nabla^*$, see \cite{Lau}.
 Moreover, it is well known that \textit{projective flatness} is 
 a fundamental notion in projective differential geometry \cite{Ei}. 
 The projective flatness is defined on the equivalence classes of affine connections on a
 manifold $M$, see Definition \ref{def:pflat}.
 A torsion-free projective flat connection $\nabla$ on a manifold  
 defines a projective structure and it has been called 
 a \textit{$\RPn$-structure}. 

  In this paper, we will first show that a statistical manifold is constant curvature 
 if and only if it is a projectively flat conjugate symmetric manifold 
(or a projectively flat conjugate Ricci symmetric manifold, that is, 
 $\Ric = \Ric^*$), Theorem \ref{thm:Charactconst}.
 Moreover, we will discuss a natural family of affine connections, 
 the so-called \textit{$\alpha$-connections} introduced in information 
 geometry \cite{AN} and 
 will show that almost all members of $\alpha$-connections 
 are not constant curvature, Proposition 
 \ref{prp:alphafamily}.

 Next we recall the basic definitions of properly convex 
 $\RPn$-structures,  
 see \cite{Gold, Lab} for more details.
 For a $\RPn$-structure (induced from a projectively flat connection $\nabla$) 
 on a manifold $M$,
  the \textit{holonomy representation}
 $\rho$ can be defined, which is a map from the fundamental 
 group $\pi_1 (M)$ to the projective group 
 $\mathrm{PSL}_{n+1} \mathbb R$.
 Accordingly the \textit{developing map} follows, which is a local diffeomorphism $f$
 from the universal cover $\widetilde M$ taking values in $\RPn$ such that 
 it is $\rho$-equivariant, that is, for any $x \in \widetilde M$ and any $\gamma \in 
 \pi_1 (M)$, 
 $f(\gamma x) = \rho(\gamma) f(x)$
 holds.
 Finally a $\RPn$-structure is called \textit{convex} if 
 the developing map is a homeomorphism to a convex set in $\RPn$ and it is 
 called \textit{properly convex} if this 
 convex set is included in a compact convex set of an affine chart. 
 It has been known that (properly) convex $\RPn$-structures 
 are fundamental 
 and important geometric structures on a manifold.

 We will then characterize the \textit{properly convex $\RPn$-structures} on a compact 
 manifold $M$, see Theorem \ref{thm:projflat}, that is, 
 such structures can be constructed from constant curvature $-1$ statistical 
 structures on $M$.  This theorem is a 
 reformulation of \cite[Theorem 3.2.1]{Lab}, in terms of 
 a constant statistical structure on $M$.
 In Appendix \ref{sc:pcflat}, we will give a sketch of the proof of Theorem 
 \ref{thm:projflat}.

 It is known that notion of statistical manifolds has also grown out of theory of 
 affine immersions. The Codazzi equation of an 
 affine immersion defines a statistical structure, and many of studies of 
 statistical manifolds relate to it, for example, constant curvature 
 statistical manifolds give affine spheres, see \cite{Kur}.
 However, in this paper, we emphasize an intrinsic characterization of 
 a constant curvature statistical manifold in terms of properties of 
 the affine connection $\nabla$ and its curvature 
 without any relation to affine immersions. 
 As a result of it, we can see a direct connection between 
 constant curvature statistical structures and 
 properly convex $\RPn$-structures on a manifold.
 
\textbf{Acknowledgement:} We would like to thank Prof. Jun-ichi Inoguchi 
 and Porf. Hitoshi Furuhata for comments on the manuscripts and 
 letting us know several related references. Moreover, we would  like to 
 thank the four anonymous reviewers for their suggestions, comments and additional 
 references.

%
%

\section{Tensor analysis for a statistical manifold}
 In this section, we recall the basic facts about a statistical manifold, see 
  for examples \cite{AN, CU, Op, NS, BNS, USA}. 
\subsection{Preliminaries}
 Let $\nabla$ be a torsion-free affine connection on 
 a pseudo-Riemannian  manifold $(M, g)$ and let $\nabla^*$ be the 
 dual torsion-free affine connection in the sense of \eqref{eq:dualnablas} . 
 Then it is easy to see that 
\[
 - C (X, Y, Z) = (\nabla_X^* g) (Y, Z)
 \] 
  holds and 
  thus $(g, \nabla)$ is a statistical structure if and only if $(g, \nabla^*)$ is.
 For a statistical manifold $(M, g, \nabla)$, we define a tensor field $K$ of type $(1, 2)$ by 
\begin{equation}\label{eq:K}
K(X, Y)  =  \nabla_X  Y- \widehat \nabla_X Y
 \end{equation}
 and it will be called the \textit{difference tensor},
 where $\widehat \nabla$ is the Levi-Civita connection of the pseudo-Riemannian 
 metric $g$.
 Moreover, defining the $(1, 1)$-tensor $K_X$ by 
 $K_X Y = K(X, Y)$, we have
 \begin{equation}\label{eq:K2}
 K_X = \nabla_X  - \widehat \nabla_X.
 \end{equation}
 Since $\nabla$ and $\widehat \nabla$ are torsion-free, thus 
 $K(X, Y)$ is symmetric.
 Then from the compatibility of $g$, that is $\widehat \nabla g = 0$,  
 we have $(\nabla_X g)(Y, Z) = (K_X g)(Y, Z)$. We also compute
 \begin{equation*}
(K_X g) (Y, Z)  =  - g (K_X Y, Z) - g (Y,  K_X Z).
\end{equation*}
From the definition of $C$ in \eqref{eq:C} and the symmetry of $C$, we
have 
\begin{equation}\label{eq:relationCK}
    C (X, Y, Z)  = -2 g(K(X, Y), Z).
\end{equation}
 The relation $- C (X, Y, Z) = (\nabla_X^* g) (Y, Z)$ also implies 
\begin{equation}\label{eq:dual}
    K_X    = - \nabla_X^* + \widehat \nabla_X.
\end{equation}
 Therefore the Levi-Civita connection $\widehat \nabla$ is the mean of $\nabla$ and its dual $\nabla^*$:
 \[
  \widehat \nabla_X Y = \frac12 (\nabla_X Y  + \nabla_X^* Y).
 \]
   Moreover,  from \eqref{eq:K} and $\widehat \nabla g = 0$, it is easy to see that 
 \begin{equation}\label{eq:covderi}
 (\widehat\nabla_X C) (Y, Z, W) = -2 g ( (\widehat \nabla_X K) (Y, Z), W )
 \end{equation}
 holds.  We now characterize the total symmetry of the covariant derivative $\widehat 
\nabla C$ 
 of the cubic form $C$ as follows.
 \begin{lemma}[Lemma 1 in \cite{BNS}]\label{lem:Csymmet}
  For a statistical manifold $(M, g, \nabla)$, the followings are mutually equivalent$:$
 \begin{enumerate}
 \item  $\nabla C $ is totally symmetric. \label{itm:nablaCsym}
 \item  $\widehat \nabla C $ is totally symmetric. \label{itm:hatnablaCsym}
 \item  $\widehat \nabla K  $ is totally symmetric. \label{itm:hatnablaKsym}
 \end{enumerate}
 \end{lemma}
 \begin{remark}
 Lemma \ref{lem:Csymmet} has been proved in \cite{BNS} for the statistical structure 
 induced by Blaschke hypersurfaces of affine differential geometry.
 The proof can be easily generalized 
 to any statistical manifold. In statistical manifolds point of view, 
 Blaschke hypersurfaces correspond to the trace-free structures of 
 difference tensors $K$, see \eqref{eq:tracefree}.
 \end{remark}
 Let $R$ denote the curvature tensor field of the connection $\nabla$, that is,
 \begin{equation}\label{eq:Curvaturetensor}
  R (X, Y) Z = \nabla_X \nabla_Y Z - \nabla_Y \nabla_X Z - \nabla_{[X, Y]} Z,
 \end{equation}
 and let $R^*$ denote the curvature tensor field of its dual connection $\nabla^*$.
 Moreover, let $\widehat R$ denote the curvature tensor field of the Levi-Civita connection 
 $\widehat \nabla$.
  \begin{lemma}[Proposition 9.1 in \cite{NS}]\label{lem:Curvatures}
  For a statistical manifold  $(M, g, \nabla)$ the following identities hold$:$
   \begin{align}\label{eq:R}
R (X, Y) &= \widehat R (X, Y)  +(\widehat \nabla_X K)_Y - (\widehat \nabla_Y K)_X + [K_X, K_Y],       \\
\nonumber & = \widehat R (X, Y)  +(\nabla_X K)_Y - (\nabla_Y K)_X - [K_X, K_Y],       \\
\label{eq:Rstar}
R^* (X, Y) &= \widehat R (X, Y)  -(\widehat \nabla_X K)_Y + (\widehat \nabla_Y K)_X + [K_X, K_Y], \\
\nonumber & = \widehat R (X, Y)  -(\nabla_X K)_Y + (\nabla_Y K)_X +3 [K_X, K_Y].
   \end{align}
   Moreover, the following identities also hold$:$
   \begin{align}
\frac12R(X, Y)- \frac12R^*(X, Y) &=(\widehat \nabla_X K)_Y - (\widehat \nabla_Y K)_X ,   
\label{eq:RicRic*0} \\
 \nonumber &= (\nabla_X K)_Y - (\nabla_Y K)_X -2[K_X, K_Y],    \\
\frac12R(X, Y)+ \frac12R^*(X, Y) &= \widehat R (X, Y) + [K_X, K_Y].
   \end{align}
  \end{lemma}

 We now assume that $M$ is orientable and 
 take the pseudo-Riemannian volume form $\omega_g $ 
 on $(M, g)$:  
 \[
 \omega_g = \sqrt{|\det g|}\, dx_1 \wedge \cdots \wedge dx_n.
\]
 Then it is easy to see that the covariant derivative $\nabla_X$ of $\omega_g$ as
\[
\nabla_X \omega_g = \frac12 \tr_g (\nabla_X g) (\cdot,\cdot) \omega_g. 
\]
 From the symmetry of $C$ we see that $\tr_g (\nabla_X g) (\cdot,\cdot)  = \tr_g (\nabla_\cdot g) (\cdot,X)$,
  and the relation \eqref{eq:relationCK} and the self-adjointness of $K$ imply that  
\[  
\frac12 \tr_g (\nabla_\cdot g) (\cdot, X) = - g(\tr_g K(\cdot, \cdot), X) =  
- \tr K_X, 
\]
 and therefore we have 
\begin{equation}\label{eq:tau}
\nabla_X \omega_g = -\tau_g(X)  \omega_g 
\quad \mbox{with}\quad  \tau_g(X) =\tr K_X.
\end{equation}
 The Ricci curvature tensor $\Ric$ of $\nabla$ is defined by 
\[
 \Ric (Y, Z) =\tr \{ X \mapsto R(X, Y) Z\}.
\]
 Similarly, the Ricci curvature tensor  $\Ric^*$ (resp. $\widehat \Ric$)
 of $\nabla^*$ (resp. $\widehat \nabla$) can be defined analogously.
 \begin{lemma}[Section 3 in \cite{Op2}]\label{lem:Ric}
  For an orientable statistical manifold $(M, g, \nabla)$ with the $1$-form $\tau_g$ in \eqref{eq:tau},
  the following identities hold$:$
 \begin{align*}
 \Ric (Y, Z) & = \widehat \Ric (Y, Z) +(\operatorname{div}^{\widehat \nabla} K) (Y, Z)
   - (\widehat \nabla_Y \tau_g) (Z) + \tau_g (K_Y  Z) - g(K_Y, K_Z),  \\
   & = \widehat \Ric (Y, Z) +(\operatorname{div}^{\nabla} K) (Y, Z)
   - (\nabla_Y \tau_g) (Z) - \tau_g (K_Y  Z) + g(K_Y, K_Z), \\
\Ric^* (Y, Z) & = \widehat \Ric (Y, Z) - (\operatorname{div}^{\widehat \nabla} K) (Y, Z)
  +(\widehat \nabla_Y \tau_g) (Z) + \tau_g (K_Y Z) - g(K_Y, K_Z), \\
  &= \widehat \Ric (Y, Z) - (\operatorname{div}^{\nabla} K) (Y, Z)
  +(\nabla_Y \tau_g) (Z) + 3 \tau_g (K_Y Z) - 3 g(K_Y, K_Z).
 \end{align*}
  Moreover, the following identities also hold$:$
 \begin{align}
\frac12 \Ric (Y, Z)  -\frac12 \Ric^* (Y,Z&) \label{eq:RicRic*}  
 = (\operatorname{div}^{\widehat \nabla} K) (Y, Z) 
 - (\widehat \nabla_Y \tau_g) (Z), \\
 & = (\operatorname{div}^{\nabla} K) (Y, Z)
 - (\nabla_Y \tau_g) (Z)-2 \tau_g(K_Y Z) +2 g(K_Y, K_Z),\nonumber 
\end{align}
\begin{align}\hspace{-2cm}
 \frac12\Ric (Y, Z)  +\frac12\Ric^* (Y, Z) =\widehat \Ric (Y, Z) + \tau_g (K(Y, Z)) - g(K_Y, K_Z).
 \end{align}
\end{lemma}

  From Lemma \ref{lem:Ric}, it is clear that the Ricci curvature of a statistical 
  manifold is not symmetric in general. 
 We now recall that a torsion-free affine connection 
  is called \textit{locally equiaffine} (resp. \textit{equiaffine})
 if there exists a volume form $\omega$
 around any point on $M$ (resp. a volume form $\omega$ on $M$)
 such that $\nabla \omega =0$. In particular if the volume form
 is the pseudo-Riemannian volume 
  form $\omega_g$, then $\nabla \omega_g =0$ is equivalent to 
\begin{equation}\label{eq:tracefree}
 \tr K_X =0,
\end{equation}
 and the condition \eqref{eq:tracefree} will be called \textit{trace-free}.
  \begin{corollary}\label{cor:Ric}
 The Ricci curvature of the connection $\nabla$ of a statistical manifold $(M, g, \nabla)$ is symmetric 
  if and only if it is locally equiaffine.
  \end{corollary}
 We next recall the projective equivalence of affine connections. 
\begin{definition}\label{def:pflat}
 Two torsion-free locally equiaffine  connections $\nabla$ and $\overline{\nabla}$ on $M$
  are called {\it projectively equivalent} if there exists a closed $1$-form $\rho$ such that 
 \[
  \overline{\nabla}_X Y = \nabla_X Y + \rho (X) Y + \rho(Y) X
 \]
 holds. In particular when $\overline{\nabla}$ is flat, then $\nabla$
  is called {\it projectively flat}. Moreover the {\it projectively 
  curvature tensor} $P$ is defined as 
\begin{equation}\label{eq:Wyel}
P (X, Y) Z = R(X, Y) Z - \frac1{n-1} \left\{ \Ric\, (Y, Z) X - \Ric\, (X, Z) Y\right\}.
\end{equation}
\end{definition}
 It is classically known  that 
 that the projectively flat connections has been characterized as follows:
\begin{theorem}[p.88 and p.96 in \cite{Ei}]\label{thm:Eisen}
 A torsion-free locally equiaffine  connection $\nabla$ on $M$ is projectively flat if 
 and only if the following condition holds$:$
 \begin{enumerate}
     \item If $\dim M = 2$, $\nabla \Ric$ is totally symmetric. In this case 
     the projective curvature tensor $P$ automatically vanishes.
      \item If $\dim M \geq 3$,  the projective curvature tensor $P$ identically vanishes.
     In this case $\nabla \Ric$ is automatically totally symmetric.

 \end{enumerate}
\end{theorem}
 We also collect the basic identities for curvatures by 
 a straight forward computation.
 \begin{lemma}\label{lem:identities}
The following identities hold$:$
\begin{enumerate}
    \item $R(X, Y) = - R(Y, X)$ and $R^*(X, Y) = - R^*(Y, X)$.
    \item $g(R(X, Y) Z,  W) =   - g(Z , R^* (X, Y) W)$.
    \item $R(X, Y) Z + R(Y, Z) X + R (Z, X) Y =0$\quad   
     {\rm $($$1$st Bianchi  identity$)$},
    \item $(\nabla_X R)(Y, Z) + (\nabla_YR)(Z, X) + (\nabla_ZR) (X, Y) =0 $\quad   {\rm 
     $($$2$nd Bianchi identity$)$}.

\end{enumerate}
 \end{lemma}

\section{Characterization of constant curvature statistical manifolds}
 In this section, we characterize constant curvature statistical manifolds 
  in terms of various conjugate symmetries and the projective flatness of 
  the dual connection.
 \subsection{Conjugate symmetries}
 We first define various conjugate symmetries for a statistical manifold.
\begin{definition}
 Let $(M, g, \nabla)$ be a statistical manifold, and let $R$ and $R^*$ denote 
 curvatures of the connection $\nabla$ and its dual 
 connection $\nabla^*$, respectively.
 Moreover, let $\Ric$ and $\Ric^*$ denote the Ricci curvatures of $\nabla$ and 
 $\nabla^*$, respectively.
 Then $(M, g, \nabla)$ will be respectively called a \textit{self-dual}, 
 a \textit{conjugate symmetric} or a \textit{conjugate Ricci-symmetric} 
 if 
 \begin{equation}\label{eq:conjugatesymmetric}
     \nabla = \nabla^*, \quad      R = R^*\quad\mbox{or}\quad      \Ric = \Ric^*
 \end{equation}
 holds.
\end{definition}
 Note that notion of ``self-dual'' and ``conjugate symmetric''
 had been introduced in \cite{A} and \cite{Lau}, respectively.
 By \eqref{eq:dualnablas}, a self-dual manifold is nothing but $\nabla$ is the Levi-Civita connection of $g$. Moreover, by \eqref{eq:C} and \eqref{eq:K2},
 the conditions $\nabla = \nabla^*$, $C= 0$ and $K=0$ are mutally equivalent.
 From now on we always assume that a manifold $M$ is orientable.
 Then we have the following.
\begin{lemma}
 The condition $\nabla = \nabla^*$ $($resp. $R = R^*)$ of 
 a statistical manifold $(M, g, \nabla)$ implies $R = R^*$  
 $($resp. $\Ric = \Ric^*)$. 
 Moreover, $\Ric= \Ric^*$ implies $\Ric (X, Y) = \Ric (Y, X)$.
\end{lemma}
\begin{proof}
 The first statements follow immediately from the definition. We now assume
 $\Ric = \Ric^*$.
 From the formulas in Lemma \ref{lem:Ric}, we have
 \begin{equation}\label{eq:Ricciconj}
0 = \Ric (Y, Z) - \Ric^* (Y, Z) =  2 (\operatorname{div}^{\widehat \nabla} K) (Y, Z) - 2  (\widehat \nabla_Y \tau_g) (Z).
 \end{equation}
 In particular, $(\widehat \nabla_Y \tau_g) (Z)$ is symmetric with respect to $Y$ and $Z$, and thus 
 $ d \tau_g (Y, Z) = (\widehat \nabla_Y \tau_g) (Z) -(\widehat \nabla_Z \tau_g) (Y)=0$ follows.
From the formulas in Lemma \ref{lem:Ric} and symmetries of $\widehat \Ric (Y, Z)$ and 
$ (\operatorname{div}^{\widehat \nabla} K) (Y, Z)  + \tau_g (K(Y, Z)) - g(K_Y, K_Z)$ 
 with respect to $Y$ and $Z$, it is easy to see that 
\begin{equation}\label{eq:Ricsymm}
 \Ric (Y, Z) - \Ric (Z, Y) =  - (\widehat \nabla_Y \tau_g) (Z) + (\widehat \nabla_Z 
 \tau g )(Y)  = - d\tau_g (Y, Z).
\end{equation}
 This completes the proof.
\end{proof}
 From \eqref{eq:Ricciconj}, we have the following corollary.
\begin{corollary}
 The followings are mutually equivalent$:$
\begin{enumerate}
\item $\Ric = \Ric^*$.
\item $(\operatorname{div}^{\widehat \nabla} K) (Y, Z) =
  (\widehat \nabla_Y \tau_g) (Z)$.
\end{enumerate}
\end{corollary}
 We next characterize the conjugate symmetry of the curvature $R$.
\begin{proposition}\label{prp:conjugate}
 The followings are mutually equivalent$:$
\begin{enumerate}
\item $R = R^*$.
\item $\nabla C$ is totally symmetric.
\item $\widehat \nabla C$ is totally symmetric.
\item $\widehat \nabla K$ is totally symmetric.
\end{enumerate}
\end{proposition}
\begin{proof}
The equivalences of (2), (3) and (4) follow from Lemma \ref{lem:Csymmet}.
The equivalence of (1) and (4) follows from the formula in \eqref{eq:RicRic*0}. 
\end{proof}

\subsection{Characterization of constant curvature statistical manifolds}
 The notion of constant curvature statistical manifold was introduced in  \cite{Kur},
 see \eqref{eq:constantcurv}.
 Then the following lemma (Schur's lemma) can be used for the characterization of 
 a constant curvature statistical manifold.
\begin{lemma}\label{lem:SY}
 Let $S$ be the $(1, 1)$-tensor $S$ defined by 
 $\Ric (X, Y) = (n-1) g (S X, Y)$. Assume that $\Ric$ and 
 $\nabla \Ric$ are totally symmetric. Then if there exists a 
 smooth function $\lambda$ such $S Y = \lambda Y$ for 
 any vector field $Y$, then $\lambda$ is constant.
\end{lemma}
\begin{proof}
 Using the definition of $\nabla_X \Ric$ and
 the symmetry of $\Ric$ we compute 
 \begin{align*}
(\nabla_X \Ric) (Y, Z) & = X \Ric ( Y, Z) - \Ric ( \nabla_X Y, Z) -  \Ric ( Y, \nabla_X Z), \\
 & = (n-1)\left\{X g( SY, Z) - g( S Y, \nabla_X Z)- g (S Z, \nabla_X Y)\right\}, \\
 & = (n-1)\left\{g(\nabla^*_X (\lambda Y), Z) - g (\lambda Z, \nabla_X Y)\right\}.
 \end{align*}
 By using the total symmetry of $\nabla \Ric$
 we compute
 \begin{align*}
     0 & = \frac1{n-1}(\nabla_X \Ric)(Y, Z ) - \frac1{n-1}(\nabla_Y \Ric)(X, Z), \\
       & = g (\nabla^*_X (\lambda Y), Z) - g (\lambda Z , \nabla_X Y) 
        -  g (\nabla^*_Y (\lambda X), Z) + g (\lambda Z , \nabla_Y X),   \\ 
        &= g((X \lambda) Y - (Y \lambda ) X, Z).
 \end{align*}
 Thus we obtain $(X \lambda) Y - (Y \lambda ) X=0$ and 
 $X \lambda = 0$ follows, that is, $\lambda$ is constant.
\end{proof}
 We now characterize a constant curvature statistical manifold.
\begin{theorem}\label{thm:Charactconst}
 For a statistical manifold $(M, g, \nabla)$, the followings are mutually equivalent$:$
\begin{enumerate}
    \item It has a constant curvature.
    \item It is a projectively flat conjugate symmetric manifold.
    \item It is a projectively flat conjugate Ricci symmetric manifold.
\end{enumerate}
\end{theorem}
\begin{proof}
(1) $\Rightarrow$ (2): 
 From the formula $(2)$ in Lemma \ref{lem:identities} and the constancy 
of the curvature, that is, 
\[
R(X, Y)Z = k \left\{ g(Y, Z) X- g(X, Z)Y\right\}
\]
for some constant $k$,
we have 
\begin{align*}
g(R^*(X, Y) Z, W) &= -g(Z, R(X, Y)W),  \\  
&=- g(k \left\{ g(Y, W) X - g (X, W) Y\right\}, Z), \\ 
&=  g(k \left\{ g(Y, Z) X - g (X, Z) Y\right\}, W), \\
&= g(R(X, Y) Z, W).
\end{align*}
Since $W$ is arbitrary, thus the conjugate symmetry of 
 the curvature $R$ follows.

The Ricci curvature for a constant curvature statistical manifold 
is easily computed as 
\begin{equation}\label{eq:Ricexp}
 \Ric (Y, Z) = (n-1) k g(Y, Z),
\end{equation}
 and therefore the projective curvature $P$ in \eqref{eq:Wyel}
 vanishes.  Thus for $n \geq 3$, $\nabla$ is projectively flat, 
 see Theorem \ref{thm:Eisen}.
 Since $R = R^*$, $\nabla^*$ is also projectively flat.

 For $n=2$, the projective flatness of $\nabla$ is equivalent to 
 the total symmetry of $\nabla \Ric$, however from \eqref{eq:Ricexp}
 it is equivalent to the total symmetry of $\nabla g$, which 
 is of course true by the definition of the statistical manifold $(M, g, \nabla)$.
 Then the projective flatness of $\nabla^*$ follows from the duality.

 (2) $\Rightarrow$ (3) is clear. 
 
(3) $\Rightarrow$ (1): 
 We define a
 $(1, 1)$-tensor $S$ by $\Ric (X, Y) =(n-1) g (S X, Y)$, then 
 from the symmetry of $\Ric$,  $g (S X, Y) = g(X, SY)$ follows.
 Moreover, the projective flatness of $\nabla$ implies that 
 the projective curvature tensor $P$ vanishes, that is,
 \begin{equation}\label{eq:R1}
  R (X, Y) Z = g(SY, Z)X - g(S X, Z) Y.
\end{equation}
 By using the formula (2) in Lemma \ref{lem:identities} and 
 the equation above, we compute
\begin{align*}
g (R^* (X, Y) Z, W ) & = -g (Z, R (X, Y) W), \\
& =-g (Z,g(SY, W)X - g(S X, W) Y), \\
& =g (g(Z, Y) S X - g(Z, X) SY, W).
\end{align*}
 Since $W$ is an arbitrary, thus we have
 \begin{equation}\label{eq:R2}
 R^* (X, Y) Z = g(Y, Z) S X - g(X, Z) SY.
 \end{equation}
 Taking the trace with respect to $X$ for the above equation, we have 
\begin{equation}\label{eq:R3}
  \Ric^* (Y, Z) = g(Y, Z) \tr S - g(S Y, Z).
\end{equation}
 Then the conjugate symmetry of $\Ric$ implies that $n g(Y, S Z) = g(Y, \tr S \cdot Z)$,
 thus $S = (\tr S) {\rm id}$ follows, and it is equivalent to 
  \begin{equation}
      S Y = \lambda Y,
  \end{equation}
 where $\lambda = \tr S/n$.
 The projective flatness of $\nabla$ also implies that 
 $\nabla \Ric$ is totally symmetric. Thus Lemma \ref{lem:SY}
 implies that $\lambda$ is constant and thus $R$ is of constant curvature.
\end{proof}
\begin{remark}
 It should be remarked that (1) $\Rightarrow$ (2) part of 
 Theorem \ref{thm:Charactconst} has been 
 proved in \cite[Theorem 2]{Rylov2}, \cite[Theorem 3.3]{MS} or 
 \cite[Theorem 9.7.2]{CU}, however,  (3) (or (2))  $\Rightarrow$ (1) part 
 has not been proved in those references.
 Moreover, the equivalence of (2) and (3) have been noted by anonymous referee
 of the paper \cite{KO}.
\end{remark}

\subsection{Conjugate symmetric $\alpha$-connections}
For a statistical manifold $(M, g, \nabla)$, it is natural to consider a family of 
statistical structure $(g, \nabla^{\alpha})$ as follows:
\begin{equation}\label{eq:alphaconn}
\nabla^{\alpha} = \widehat \nabla + \alpha K, 
\end{equation}
where $\alpha$ is a real constant. Note that $\nabla^{1}= \nabla$, and 
 from \eqref{eq:dual} we have
 $\nabla^{-1} = \nabla^{*}$. Moreover, 
 \[
  \nabla^{-\alpha} = (\nabla^{*})^{\alpha}
 \]
 holds.
The family of connections $\{\nabla^{\alpha}\}_{\alpha \in \mathbb R}$ has been 
 called the \textit{$\alpha$-connections}, see \cite{AN}.
 It is easy to see that $(g, \nabla^{\alpha})$ is a statistical structure 
 for any $\alpha \in \mathbb R$, and if $\nabla$ (resp. $R$ or $\Ric$) is 
 conjugate symmetric, then  $\nabla^{\alpha}$ (resp. $R^{\alpha}$ or $\Ric^{\alpha}$) 
 is conjugate symmetric for any $\alpha \in \mathbb R$.

The following proposition is a slight generalization of Proposition 4.1 in \cite{Op2}.
\begin{proposition}\label{prp:alphafamily}
 Let $(M, g, \nabla)$ be a constant curvature statistical manifold, and 
 let  $\nabla^{\alpha}$ the $\alpha$-connection in \eqref{eq:alphaconn} 
 for $\alpha \in \mathbb R$.
 Assume that  the pseudo-Riemannian metric $g$ is not constant curvature. Then 
 the statistical manifold $(M, g, \nabla^{\alpha})$ does not have constant curvature 
 except $\alpha = \pm 1$, that is the cases for 
 $\nabla^1= \nabla$ and $\nabla^{-1}=\nabla^*$.
\end{proposition}
\begin{proof}
 From Theorem \ref{thm:Charactconst}, the constant curvature property of $(M, g,\nabla)$
 is equivalent to the conjugate symmetry of $R$ and the projective flatness of $\nabla^*$. 
 By the conjugate symmetry of $R$, we have
 \begin{equation}\label{eq:RhatR}
  R(X, Y) = \widehat R (X, Y) +[K_X, K_Y].
  \end{equation}
  Moreover, the projective flatness of $\nabla^*$ can be written as 
\[
 R(X, Y) Z = R^* (X, Y) Z =  \frac1{n-1}
\left\{
\Ric (Y, Z) X - \Ric (X, Z) Y
 \right\}.
\]
Then since 
$\nabla^{\alpha} = \widehat \nabla + \alpha K$, 
we have  
 \begin{equation}\label{eq:alphacurvature}
  R^{\alpha}(X, Y) = \widehat R (X, Y) +\alpha^2[K_X, K_Y] 
   = R (X, Y) +(\alpha^2-1)[K_X, K_Y]. 
  \end{equation}
 It follows that for $\alpha \neq \pm 1$
\begin{equation}\label{eq:idcurv}
 \widehat R(X, Y) = \frac1{1- \alpha^2} 
 (R^{\alpha}(X, Y)- \alpha^2 R(X, Y))
\end{equation}
 holds. Then the claim follows.
 \end{proof}
\begin{remark}
  The assumption of non-constant curvature of the pseudo-Riemannian metric 
  $g$ in Proposition \ref{prp:alphafamily} 
  is a mild restriction. In fact, many Fisher metrics determined from 
  statistics are not constant curvature, for examples, the multivariate 
  normal distributions and the gamma distributions etc, see \cite{Lau}. 
\end{remark}
We finally give an example of flat statistical manifolds.
\begin{example}[Multivariate normal distributions]
 The density function $p$ of the multivariate normal distribution
 is given by 
\begin{align*}
    p(x,\mu,\Sigma)=\frac{1}{\sqrt{(2\pi)^n}\det (\Sigma)}\exp \left\{-\frac{1}{2}(x-\mu)^T\Sigma^{-1}(x-\mu) \right\},
\end{align*}
 where $\mu \in \mathbb R^n$ is the mean and 
 $\Sigma$ is the covariance, which takes values in the set of degree $n$
 positive definite matrices.
 Therefore the set of density functions is parameterized by 
\begin{align*}
    \mathcal{N}=\operatorname{Pos}(n,\mathbb{R})\times \mathbb{R}^n.
\end{align*}
 Since the affine group $\operatorname{Aff}(n, \mathbb R)= \operatorname{GL}_n \mathbb R 
 \ltimes \mathbb R^n$ acts transively on $\mathcal{N}$, it can be represented by 
 a homogeneous space $\operatorname{Aff}(n, \mathbb R)/\operatorname{Stab}$,
 where $\operatorname{Stab}$ denotes the stabilizer at a base point, see \cite{GQ},
 thus $\mathcal N$ has a manifold structure.
 We then define the \textit{Fisher metric} and the cubic tensor
 of $\mathcal N$ as
\begin{align*}
 g^F(X,Y)&=E[(X\log p) (Y \log p) ], \\
 C(X, Y, Z)&=E[(X\log p) (Y \log p) (Z \log p)], 
\end{align*}
where $X$, $Y$, $Z$ are tangent vectors of $\mathcal N$,
and \[
E[f]=\int_{\mathbb R^n}
 f(x)p(x, \mu, \Sigma)\, d x 
    \]
for an integrable function $f$ on $\mathbb{R}^n$, see 
\cite{AN}. 
For any real constant $\alpha$, we define an affine connection 
$\nabla^{\alpha}$ by 
\begin{equation*}
 g^F(\nabla^{\alpha}_X Y,Z)
  =g^F(\nabla^{g^F}_X Y, Z)-\dfrac{\alpha}{2} C(X,Y,Z), 
\end{equation*}
 where $\nabla^{g^F}$ denotes the Levi-Civita connection 
of the Riemannian metric $g^F$. 
 The affine connection $\nabla^{\alpha}$ has been called the
 the \emph{Amari-Chentsov $\alpha$-connection} 
 for the space of normal distributions.

 Then a straightforward computation shows that all Christoffel symbols $\Gamma_{ijk}$
 are zero for $\alpha =1$ and by the duality, the flatness of $\nabla^{\pm 1}$ follows. 

 For $n=1$, the Fisher metric $g^F$ can be computed as 
\[
 I = \begin{pmatrix} \frac{1}{s} & 0\\0 & \frac1{2 s^2 } \end{pmatrix},
\]
 where $\Sigma = s>0$.
 It is easy to see that $g^F$ gives 
 a constant curvature hyperbolic metric, thus the Riemannian curvature 
 $\widehat R$ is constant $-1/2$. Applying the formula in \eqref{eq:idcurv}, we have 
 the constant curvature $-(1 - \alpha^2)/2$ for $R^{\alpha}$.

 On the other hand for $n\geq 2$, the metric connection $\nabla^{g^F}$ does not have
 constant curvature, thus by Proposition \ref{prp:alphafamily}, the
 curvature of $R^{\alpha}$ is not constant except $\alpha = \pm 1$.
 However since $R^{\pm 1}=0$, $R^{\alpha}$ satisfies 
 $R^{\alpha} = R^{-\alpha}$ for any $\alpha \in 
 \mathbb R$. 
 
\end{example}
 \begin{remark}
\mbox{}
\begin{enumerate}
\item  The dual flat structures on a statistical manifold is 
 particularly important for information geometry, see \cite{AN}.

 \item 
 It is known that the Fisher metric $g^F$ and the Amari-Chentsov $\alpha$-connection 
 $\nabla^{\alpha}$ of $\mathcal N$
 is invariant under a Lie group action of $\operatorname{Aff}(n, \mathbb R)$
 and moreover, a subgroup of $\operatorname{Aff}(n, \mathbb R)$ acts simply and transively 
 on $\mathcal N$. Thus $(\mathcal N, g^F, \nabla^{(\alpha)})$ 
 becomes a \textit{statistical Lie group}, see Definition 1.4.1 in \cite{FIK}.
 
\item More generally for a given Lie group $G$ with a left-invariant metric $g$,
 one can construct a left-invariant 
 affine connection $\nabla$ on $G$ such that the curvature of $\nabla$ is constant, 
 that is, $(G, g, \nabla)$ is a statistical Lie group with constant curvature, 
 as follows:
 A left-invariant connection $\nabla$ on $G$ is given by values of Christoffel symbols
(which are constant) at 
 a base point and by Theorem \ref{thm:Charactconst}. Then constancy of 
 the curvature of $\nabla$ is characterized by $R = R^*$
 and the projective flatness of $\nabla$.
 Moreover $R=R^*$ is equivalent to the total 
 symmetry of $\widehat \nabla C$
 by Proposition \ref{prp:conjugate}. Thus constancy of the curvature can be 
 determined by algebraic equations among Christoffel symbols.  
 It is not difficult to see that there are many solutions
 in general.
\end{enumerate}
\end{remark}
 
\section{Projectively flat structures and constant curvature statistical structures}
\label{sc:PFS}
 In this section we will apply the characterization of constant curvature 
 statistical structures to properly convex $\RPn$-structures on 
 a compact Riemannian manifold.

 In \cite{Lab}, Labourie characterized properly convex $\RPn$-structures on a 
 compact manifold $M$ by 
a certain condition, which he called the 
 \textit{Condition (E)}.
 To explain it, let $g$ be a Riemannian metric on $M$, 
 and let $L = \mathbb R \times M$ the trivial bundle over $M$.
 The Condition (E) is defined as follows:
\begin{align*}
&\mbox{Condition (E)} \quad\Longleftrightarrow \quad 
\left\{
\begin{array}{l}
\mbox{(1) $\nabla$ preserves a volume.} \\
\mbox{(2) $\nabla^{g}$ is flat.}
\end{array}
\right.
\end{align*}
 Here $\nabla$ is a torsion-free affine connection on $M$, and 
 $\nabla^{g}$ is a connection on $TM \oplus L$ given by 
\begin{equation*}
  \nabla^g_X 
\begin{pmatrix} Z \\ \lambda \end{pmatrix} = 
\begin{pmatrix} \nabla_X Z + \lambda X \\ L_X \lambda + g(Z, X) \end{pmatrix}. 
\end{equation*}
 It is immediately that $\nabla^g$ is flat if and only if the following conditions hold:
\begin{equation}\label{eq:equicond}
 \left\{
\begin{array}{l}
 \nabla_X g(Y, Z) - \nabla_Y g(X, Z)=0,\\
 R(X, Y) Z -g(X, Z) Y + g(Y, Z) X=0.
\end{array}
\right.
\end{equation}
 Then the properly convex $\RPn$-structures on a compact manifold have been characterized 
 by Condition (E), see Theorem 3.2.1 in \cite{Lab},   
 which is the following Theorem \ref{thm:projflat}.

 From the discussion in Introduction, the first condition in \eqref{eq:equicond} 
 is just equivalent to that $(\nabla, g)$ defines a statistical structure, 
 see \eqref{eq:C},
 and the second condition is equivalent to the constancy of the curvature 
 $R$ with $k=-1$, see \eqref{eq:constantcurv}.
 Moreover, since constancy of the curvature implies symmetry of the Ricci curvature, 
 and by Corollary \ref{cor:Ric}, we have automatically the condition (1) in 
 Condition (E). Therefore, Theorem 3.2.1 in \cite{Lab} 
 can be reformulated as follows:
\begin{theorem}\label{thm:projflat}
 Let $(M, g, \nabla)$ be a compact constant curvature $-1$ statistical manifold 
 such that the metric $g$ is positive definite.
 Then $\nabla$ is projectively flat and defines a properly convex structure on $M$.
 Conversely, every properly convex projectively flat structure on $M$ is
 obtained in this way. 
\end{theorem}
 We will give a sketch of the proof in Appendix \ref{sc:pcflat}.
\begin{remark}\label{rm:pf}
 The projective flatness of the connection $\nabla$ 
 for a constant curvature statistical manifold $(M, g, \nabla)$
 in Theorem \ref{thm:projflat} immediately
 follows from Theorem \ref{thm:Charactconst}. Moreover, 
 the dual connection $\nabla^*$ of $\nabla$ also defines a properly 
 convex $\RPn$-structure. This induces a certain duality on the moduli space
 of properly convex $\RPn$-structures, see the discussion below 
 for the case of surfaces. 
\end{remark}
 Moreover, Labourie has characterized convex $\mathbb{RP}^2$-structures on 
 a compact surface $\Sigma$ in terms of pairs of a complex structure and a holomorphic 
 cubic differential, \cite[Theorem 1.0.2]{Lab}, which is a consequence of 
 Theorems 4.2.1 and 4.1.1. of loc. cit. 
 In the proof of Theorem 4.1.1, on page 1070, Condition (E) can be translated to four conditions
 \begin{enumerate}
 \item $A(X)$ is symmetric and trace free.
 \item $A(X) Y = A(Y) X$.
 \item $d^{\nabla} A = 0$.
 \item $R^g (X, Y)Z+[A(X), A(Y)]Z + g(Y, Z) X- g(X, Z) Y=0$, 
 \end{enumerate}
 where $g$ is a metric on $\Sigma$, $\nabla$ is the Levi-Civita connection 
 and $A(X)Y$ is the difference tensor.  In terms of our terminology, $A=K$ 
 and $\nabla = \widehat \nabla$. The first condition is that the trace-free 
 structure of $K$ in \eqref{eq:tracefree} and the third condition is 
 the total symmetry of $\widehat \nabla K$, which follows from 
 the constant curvature property of $\nabla$, see Proposition \ref{prp:conjugate}.
 The fourth condition the first formula Lemma \ref{lem:Curvatures} under 
 constancy of the curvature $R$.
 Therefore, they can be naturally interpreted 
 by notion on a statistical manifold.

 By using Theorem 4.2.1 of \cite{Lab}, it has been shown that 
 there is one-to-one correspondence between 
 a set of triads $(\nabla, \omega, J)$ ($\omega$ is a volume form and $J$ is a 
 complex structure on $\Sigma$) which   satisfies Condition (H), see 
  the equation (13) in page 1084, and 
\[
 \mathrm{Rep}_H(\pi_1(\Sigma), \mathrm{PSL}_3 \mathbb R),
\]
 where $\mathrm{Rep}_H$ denotes the holonomy representation.
 In Section 7.1, using Condition (H), a duality in the space of representations
 has been discussed. It has been shown that 
 a triad $(\nabla, \omega, J)$ satisfies Condition (H)
 if and only if $(-J\nabla J, \omega, J)$ satisfies it, thus 
 the above correspondence defines a duality on the space of representations.
 In the statistical structure point view, $-J\nabla J$ is nothing but the 
 dual connection $\nabla^*$ of $\nabla$. In fact for a non-degenerate metric 
 $g$ on $\Sigma$ such that $g$ induces the volume $\omega$, 
 $X g(Y, Z) = X g(Z, Y) = X g(J Z, JY)$ because $g$ 
 is compatible with the complex structure $J$.
 A straightforward computation shows by using \eqref{eq:dualnablas} that 
\begin{align*}
 g(\left\{\nabla_X  + J \nabla^*_X J \right\}Y, Z )+ g(Y, 
 J \left\{ \nabla_X  +  J \nabla^*_X J  \right\} J Z) =0
\end{align*}
holds. Thus $\nabla^* = - J \nabla J$ follows, and the dual 
 constant curvature statistical structure naturally 
 induces a duality on the moduli space of representations. 
\begin{remark}
 In   Theorem 4 of \cite{Lof},
 Loftin has characterized properly convex structures 
 in terms of affine sphere structures for a given $\RPn$-manifold. 
 It is evident that 
 an affine sphere structure gives a statistical structure with 
 constant curvature. On the one hand, Theorem \ref{thm:projflat} in this paper 
 has shown
 that a compact constant $-1$ statistical manifold $M$ is a $\RPn$-manifold. 
\end{remark}
\appendix 
\section{A sketch of the proof of Theorem \ref{thm:projflat}}\label{sc:pcflat}
 As explained in Section \ref{sc:PFS}, property convex $\RPn$-structures on a compact manifold $M$ have been characterized by Condition (E), and 
 we have rephrased it  in terms of constant curvature $-1$
 statistical manifold structure on $M$, and we have obtained 
 Theorem \ref{thm:projflat}.
 In this appendix, we will briefly give a proof of this theorem 
 along the proof of Theorem 3.2.1 in \cite{Lab}.
 
 Let us prove the necessary part. Let $\widetilde M$ be the universal cover of $M$
 and consider the bundle $T \widetilde M \oplus L$, where $L = \mathbb R \times \widetilde M$. Then it is evident that by the flatness of $\nabla^g$ (which is equivalent to 
 the constant curvature $-1$ statistical structure),
 $T \widetilde M  \oplus L$ is isomorphic to the trivial bundle $\mathbb R^{n+1} \times \widetilde M$.
 Let us  take the projection $p$ from $T \widetilde M \oplus L
 \cong  \mathbb R^{n+1} \times \widetilde M$ to $\mathbb R^{n+1}$.
 Moreover, let us take the canonical section of $T \widetilde 
 M \oplus L$ by $u_0 : m 
 \to (0, 1)$, and define
 $\phi = p \circ u_0$. 
 Then by the construction, $\phi$ is a $\rho$-equivariant mapping from $\widetilde M$
 to $\mathbb R^{n+1}$, where $\rho$ is 
 the holonomy representation  of the flat connection 
 $\nabla^g$.
  Then 
 Labourie has proved that $\phi$ is a proper immersion, and the image $\phi (M)$ is a locally 
 convex proper hypersurface
 in a sequence of Propositions, 
 3.2.2 (immersion), 3.2.3 (strictly locally convex and radial) and 3.2.4 (proper) 
 in \cite{Lab}, respectively.  Moreover, the geodesic $\gamma(t)$ with respect to 
 $\nabla$ gives
 a sub-bundle 
\[
 P = \mathbb R (\dot\gamma) \oplus \mathbb R \subset TM \oplus L,
\]
 which is parallel along $\gamma(t)$. Then $\phi(\gamma(t))$ is the projective line defined by 
 $P$, and $\nabla$ and $\nabla^g$ define 
 the same projective flat structure. 
 Therefore the connection $\nabla$ gives a properly convex 
 $\RPn$-structure on $M$.

 Let us prove the sufficient part.
 An important step has been proved by Vinberg \cite{Vin} and see 
 Lemma 3.1.1 in \cite{Lab}:
 For a given properly convex $\RPn$-structure on a manifold $M$ induced by 
 the pair $(f, \rho)$, there exists a proper $\rho$-equivariant immersion 
 $\tilde f$ from a universal cover $\widetilde M$ into $\mathbb R^{n+1}$
 such that the image is strictly 
 convex and radial and $\pi \circ \tilde f = f$, where $\pi$ is the projection 
 $\mathbb R^{n+1}\setminus\{0\}$ to $\RPn$.
 Note that a hypersurface is \textit{strictly convex} means that 
 it does not contain any segment and
 a hypersurface is \textit{radial} means that if the vector 
 pointing from the origin points inward.
 Let $\Sigma = \tilde f (\widetilde M)$ be the locally strictly convex hypersurface in 
 $\mathbb R^{n+1}$. Since $\Sigma$ is radial, thus 
\[
 T\mathbb R^{n+1}|_{\Sigma}= T \Sigma \oplus \mathbb R N.
\]
 The standard flat connection $\nabla^0$ on $\mathbb R^{n+1}$ then 
 induces a volume preserving connection $\nabla$ given as
\[
 \nabla_X^0 (Z + \lambda N) = \nabla_X Z + \lambda \nabla_X^0 N +
 (L_X \lambda). N + g(Z, X).N
\]
 for vector fields $X, Z$ on $\Sigma$. Since $\Sigma$ is  strictly locally
 convex and radial, $\nabla_X^0 N = X$ and $g(X, X)>0$. Now 
 the pair $(\nabla, g)$ clearly gives a constant curvature $-1$ statistical 
 manifold structure on $M$.

\bibliographystyle{plain}
\def\cprime{$'$}

\end{document}